\documentclass{amsart}
\usepackage{graphicx, amsmath, amssymb}
\usepackage{color}

\newtheorem{thm}{Theorem}[section]
\theoremstyle {definition}
\newtheorem{cor}[thm]{Corollary}
\newtheorem{prop}[thm]{Proposition}

\newtheorem{lem}[thm]{Lemma}

\numberwithin{equation}{section}

\begin{document}

\title[The annihilating-submodule graph of modules]
{The annihilating-submodule
 graph of modules over commutative rings II}%
\author{Habibollah Ansari-Toroghy }%

\author{Shokoufeh Habibi}%

\subjclass [2010] {primary
05C75, secondary 13C13}%
\keywords{ annihilating-submodule graph, cyclic module, minimal
prime submodule, chromatic and clique
number.}%

\date{\today}%
\begin{abstract}
Let $M$ be a module over a commutative ring $R$. In this paper, we
continue our study of annihilating-submodule graph $AG(M)$ which
was introduced in (The Zariski topology-graph of modules over
commutative rings, Comm. Algebra., 42 (2014), 3283--3296). $AG(M)$
is a (undirected) graph in which a nonzero submodule $N$ of $M$ is
a vertex if and only if there exists a nonzero proper submodule
$K$ of $M$ such that $NK=(0)$, where $NK$, the product of $N$ and
$K$, is defined by $(N:M)(K:M)M$ and two distinct vertices $N$ and
$K$ are adjacent if and only if $NK=(0)$. We prove that if $AG(M)$
is a tree, then either $AG(M)$ is a star graph or a path of order
4 and in the latter case $M\cong F \times S$, where $F$ is a
simple module and $S$ is a module with a unique non-trivial
submodule. Moreover, we prove that if $M$ is a cyclic module with
at least three minimal prime submodules, then $gr(AG(M))=3$ and
for every cyclic module $M$, $cl(AG(M)) \geq |Min(M)|$.
\end{abstract}
\maketitle
\section{Introduction}
 Throughout this paper $R$ is a commutative ring with a non-zero
identity and $M$ is a unital $R$-module. By $N\leq M$ (resp. $N<
M$) we mean that $N$ is a submodule (resp. proper submodule) of
$M$.

Define $(N:_{R}M)$ or simply $(N:M)=\{r\in R|$ $rM\subseteq N\}$
for any $N\leq M$. We denote $((0):M)$ by $Ann_{R}(M)$ or simply
$Ann(M)$. $M$ is said to be faithful if $Ann(M)=(0)$.

Let $N, K\leq M$. Then the product of $N$ and $K$, denoted by
$NK$, is defined by $(N:M)(K:M)M$ (see \cite{af07}).

There are many papers on assigning graphs to rings or modules
(see, for example, \cite{al99, ah14, b88, br11}). The
annihilating-ideal graph $AG(R)$ was introduced and studied in
\cite{br11}. $AG(R)$ is a graph whose vertices are ideals of $R$
with nonzero annihilators and in which two vertices $I$ and $J$
are adjacent if and only if $IJ=(0)$. Later, it was modified and
further studied by many authors (see, e.g., [1-3]).

In \cite{ah14, ah16}, we generalized the above idea to submodules
of $M$ and defined the (undirected) graph $AG(M)$,
called \textit {the annihilating-submodule graph}, with vertices\\
$V(AG(M))$= $\{N \leq M |$ there exists $(0)\neq K<M$ with
$NK=(0)$\}. In this graph, distinct vertices $N,L \in V(AG(M))$
are adjacent if and only if $NL=(0)$. Let $AG(M)^{*}$ be the
subgraph of $AG(M)$ with vertices $V(AG(M)^{*})=\{ N<M$ with
$(N:M)\neq Ann(M)|$ there exists a submodule $K<M$ with $(K:M)\neq
Ann(M)$ and $NK=(0)\}$. Note that $M$ is a vertex of $AG(M)$ if
and only if there exists a nonzero proper submodule $N$ of $M$
with $(N:M)=Ann(M)$ if and only if every nonzero submodule of $M$
is a vertex of $AG(M)$.

In this work, we continue our studying in \cite{ah14, ah16} and we
generalize some results related to annihilating-ideal graph
obtained in [1-3] for annihilating-submodule graph.

A prime submodule of $M$ is a submodule $P\neq M$ such that
whenever $re\in P$ for some
  $r\in R$ and $e \in M$, we have $r\in (P:M)$ or $e\in P$ \cite{lu84}.

The prime radical $rad_{M}(N)$ or simply $rad(N)$ is defined to be
the intersection of all prime submodules of $M$ containing $N$,
and in case $N$ is not contained in any prime submodule,
$rad_{M}(N)$ is defined to be $M$ \cite{lu84}.

The notations $Z(R)$, $Nil(R)$, and $Min(M)$ will denote the set
of all zero-divisors, the set of all nilpotent elements of $R$,
and the set of all minimal prime submodules of $M$, respectively.
Also, $Z_{R}(M)$ or simply $Z(M)$, the set of zero divisors on
$M$, is the set $\{r\in R|$ $rm=0$ for some $0\neq m\in M \}$.

A clique of a graph is a complete subgraph and the supremum of the
sizes of cliques in $G$, denoted by $cl(G)$, is called the clique
number of $G$. Let $\chi(G)$ denote the chromatic number of the
graph $G$, that is, the minimal number of colors needed to color
the vertices of $G$ so that no two adjacent vertices have the same
color. Obviously $\chi(G)\geq cl(G)$.

In section 2 of this paper, we prove that if $AG(M)$ is a tree,
then either $AG(M)$ is a star graph or is the path $P_{4}$ and in
this case $M\cong F\times S$, where $F$ is a simple module and $S$
is a module with a unique non-trivial submodule (see Theorem
\ref{t2.7}). Next, we study the bipartite annihilating-submodule
graphs of modules over Artinian rings (see Theorem \ref{t2.8}).
Moreover, we give some relations between the existence of cycles
in the annihilating-submodule graph of a cyclic module and the
number of its minimal prime submodules (see Theorem \ref{t2.18}
and Corollary \ref{c2.19})).

Let us introduce some graphical notions and denotations that are
used in what follows: A graph $G$ is an ordered triple $(V(G),
E(G), \psi_G )$ consisting of a nonempty set of vertices,
 $V(G)$, a set $E(G)$ of edges, and an incident function $\psi_G$ that associates an
 unordered pair
 of distinct vertices with each edge. The edge $e$ joins $x$ and $y$ if $\psi_G(e)=\{x, y\}$, and we
 say $x$ and $y$ are adjacent. A path in graph $G$ is a finite sequence of vertices $\{x_0,
x_1,\ldots ,x_n\}$, where $x_{i-1}$ and $x_i$ are adjacent for
each $1\leq i\leq n$ and we denote $x_{i-1} - x_i$ for existing an
edge between
 $x_{i-1}$ and $x_i$.

A graph $H$ is a subgraph of $G$, if $V(H)\subseteq V(G)$,
$E(H)\subseteq E(G)$, and $\psi_H$ is the restriction of $\psi_G$
to $E(H)$. A bipartite graph is a graph whose vertices can be
divided into two disjoint sets $U$ and $V$ such that every edge
connects a vertex in $U$ to one in $V$; that is, $U$ and $V$ are
each independent sets and complete bipartite graph on $n$ and $m$
vertices, denoted by $K_{n, m}$, where $V$ and $U$ are of size $n$
and $m$, respectively, and $E(G)$ connects every vertex in $V$
with all vertices in $U$. Note that a graph $K_{1, m}$ is called a
star graph and the vertex in the singleton partition is called the
center of the graph. For some $U\subseteq V (G)$, we denote by
$N(U)$, the set of all vertices of $G\setminus U$ adjacent to at
least one vertex of $U$. For every vertex $v\in V(G)$, the size of
$N(v)$ is denoted by $d(v)$. If all the vertices of $G$ have the
same degree $k$, then $G$ is called $k$-regular, or simply
regular. An independent set is a subset of the vertices of a graph
such that no vertices are adjacent. We denote by $P_{n}$ and
$C_{n}$, a path and a cycle of order $n$, respectively. Let $G$
and $G'$ be two graphs. A graph homomorphism from $G$ to $G'$ is a
mapping $\phi: V(G) \longrightarrow V(G')$ such that for every
edge $\{u, v\}$ of $G$, $\{\phi(u), \phi(v)\}$ is an edge of $G'$.
A retract of $G$ is a subgraph $H$ of $G$ such that there exists a
homomorphism $\phi: G \longrightarrow H$ such that $\phi(x)= x$,
for every vertex $x$ of $H$. The homomorphism $\phi$ is called the
retract (graph) homomorphism
 (see \cite{r05}).

\section{The Annihilating-submodule graph II}

An ideal $I\leq R$ is said to be nil if $I$ consist of nilpotent
elements; $I$ is said to be nilpotent if $I^{n}=(0)$ for some
natural number $n$.

\begin{prop}\label{p2.1} Suppose that $e$ is an idempotent element of
$R$. We have the following statements.

\begin {itemize}
\item [(a)] $R=R_{1}\oplus R_{2}$, where $R_{1}=eR$ and
$R_{2}=(1-e)R$. \item [(b)] $M=M_{1}\oplus M_{2}$, where
$M_{1}=eM$ and $M_{2}=(1-e)M$. \item [(c)] For every submodule $N$
of $M$, $N=N_{1}\times N_{2}$ such that $N_{1}$ is an
$R_{1}$-submodule $M_{1}$, $N_{2}$ is an $R_{2}$-submodule
$M_{2}$, and $(N:_{R}M)=(N_{1}:_{R_{1}}M_{1})\times
(N_{2}:_{R_{2}}M_{2})$.  \item [(d)] For submodules $N$ and $K$ of
$M$, $NK=N_{1}K_{1} \times N_{2}K_{2}$ such that $N=N_{1}\times
N_{2}$ and $K=K_{1}\times K_{2}$. \item [(e)]  Prime submodules of
$M$ are $P\times M_{2}$ and $M_{1}\times Q$, where $P$ and $Q$ are
prime submodules of $M_{1}$ and $M_{2}$, respectively.
\end{itemize}

\end{prop}

\begin{proof}
This is clear.
\end{proof}

We need the following lemmas.

\begin{lem}\label{l2.2} (See \cite[Proposition 7.6]{af74}.)
Let $R_{1}, R_{2}, \ldots , R_{n}$ be non-zero ideals of $R$. Then
the following statements are equivalent:

\begin{itemize}
\item [(a)] $_{R}R= R_{1} \oplus \ldots \oplus R_{n}$; \item [(b)]
As an abelian group $R$ is the direct sum of $ R_{1}, \ldots ,
R_{n}$; \item [(c)] There exist pairwise orthogonal central
idempotents $e_{1},\ldots, e_{n}$ with $1=e_{1}+ \ldots +e_{n}$,
and $R_{i}=Re_{i}$, $i=1, \ldots ,n$.
\end{itemize}

\end{lem}

\begin{lem}\label{l2.3} (See \cite[Theorem 21.28]{l91}.)
Let $I$ be a nil ideal in $R$ and $u\in R$ be such that
 $u+I$ is an idempotent in $R/I$. Then there exists an idempotent
 $e$ in $uR$ such that $e-u\in I$.
\end{lem}

\begin{lem}\label{l2.4} (See \cite[Lemma 2.4]{ah16}.)
Let $N$ be a minimal submodule of $M$ and let $Ann(M)$ be a nil
ideal. Then we have $N^{2}=(0)$ or $N=eM$ for some idempotent
$e\in R$.
\end{lem}

\begin{prop}\label{p2.5} Let $R/Ann(M)$ be an Artinian ring and let $M$ be
 a finitely generated module.
Then every nonzero proper submodule $N$ of $M$ is a vertex in
$AG(M)$.
\end{prop}

\begin{proof} Let $N$ be a non-zero submodule of $M$. So
there exists a maximal submodule $K$ of $M$ such that $N\subseteq
K$. Hence we have $(0:_{M}(K:M))\subseteq (0:_{M}(N:M))$. Since
$R/Ann(M)$ is an Artinian ring, $(K:M)$ is a minimal prime ideal
containing $Ann(M)$. Thus $(K:M)\in Ass(M)$.
 It follows that $(K:M)=(0:m)$ for some $0\neq m\in M$. Therefore $N(Rm)=(0)$, as desired.
\end{proof}

\begin{lem}\label{l2.6} Let $M=M_{1} \times M_{2}$, where
$M_{1}=eM$, $M_{2}=(1-e)M$, and $e$ $( e\neq 0, 1)$ is an
idempotent element of $R$. If $AG(M)$ is a triangle-free graph,
then one of the following statements holds.

\begin {itemize}
\item [(a)] Both $M_{1}$ and $M_{2}$ are prime $R$-modules.\item
[(b)] One $M_{i}$ is a prime module for $i=1, 2$ and the other one
is a module with a unique non-trivial submodule.
\end{itemize}
Moreover, $AG(M)$ has no cycle if and only if either $M=F \times
S$ or $M=F \times D$, where $F$ is a simple module, $S$ is a
module with a unique non-trivial submodule, and $D$ is a prime
module.
\end{lem}

\begin{proof}
If none of $M_{1}$ and $M_{2}$ is a prime module, then there exist
$r\in R_{i}$ ($R_{1}=Re$ and $R_{2}=R(1-e)$), $0\neq m_{i}\in
M_{i}$ with $r_{i}m_{i}=0$, and $r_{i}\notin Ann_{R_{i}}(M_{i})$
for $i=1, 2$. So $r_{1}M_{1} \times (0)$, $(0) \times r_{2}M_{2}$,
and $R_{1}m_{1} \times R_{2}m_{2}$ form a triangle in $AG(M)$, a
contradiction. Thus without loss of generality, one can assume
that $M_{1}$ is a prime module. We prove that $AG(M_{2})$ has at
most one vertex. On the contrary suppose that $\{N, K\}$ is an
edge of $AG(M_{2})$. Therefore, $M_{1} \times (0)$, $(0) \times
N$, and $(0) \times K$ form a triangle, a contradiction. If
$AG(M_{2})$ has no vertex, then $M_{2}$ is a prime module and so
part (a) occurs. If $AG(M_{2})$ has exactly one vertex, then by
\cite[Theorem 3.6]{ah14} and Proposition \ref{p2.5}, we obtain
part (b). Now, suppose that $AG(M)$ has no cycle. If none of
$M_{1}$ and $M_{2}$ is a simple module, then choose non-trivial
submodules $N_{i}$ in $M_{i}$ for some $i = 1, 2$. So $N_{1}
\times (0)$, $(0) \times N_{2}$, $M_{1} \times (0)$, and $(0)
\times M_{2}$ form a cycle, a contradiction. The converse is
trivial.
\end{proof}

\begin{thm}\label{t2.7}
If $AG(M)$ is a tree, then either $AG(M)$ is a star graph or
$AG(M)\cong P_{4}$. Moreover, $AG(M)\cong P_{4}$ if and only if
$M=F \times S$, where $F$ is a simple module and $S$ is a module
with a unique non-trivial submodule.
\end{thm}

\begin{proof}
If $M$ is a vertex of $AG(M)$, then there exists only one vertex
$N$ such that $Ann(M)=(N:M)$ and since $AG(M)^{*}$ is an empty
subgraph, hence $AG(M)$ is a star graph. Therefore we may assume
that $M$ is not a vertex of $AG(M)$. Suppose that $AG(M)$ is not a
star graph. Then $AG(M)$ has at least four vertices. Obviously,
there are two adjacent vertices $N$ and $K$ of $AG(M)$ such that
$|V(N)\setminus \{K\}|\geq 1$ and $|V(K)\setminus \{N\}|\geq 1$.
Let $V(N)\setminus \{K\} = \{N_{i}\}_{i\in \Lambda}$ and
$V(K)\setminus \{N\} = \{K_{j}\}_{j\in \Gamma}$. Since $AG(M)$ is
a tree, we have $V(N)\cap V(K)=\emptyset$. By \cite[Theorem
3.4]{ah14}, $diam(AG(M))\leq 3$. So every edge of $AG(M)$ is of
the form $\{N, K\}$, $\{N, N_{i}\}$ or $\{K, K_{j}\}$, for some
$i\in \Lambda$ and $j\in \Gamma$. Now,
consider the following claims:\\\\
Claim 1. Either $N^{2}=(0)$ or $K^{2}=(0)$. Pick $p\in \Lambda$
and $q\in \Gamma$. Since $AG(M)$ is a tree, $N_{p}K_{q}$ is a
vertex of $AG(M)$. If $N_{p}K_{q}=N_{u}$, for some $u\in \Lambda$,
then $KN_{u}=(0)$, a contradiction. If $N_{p}K_{q}=K_{v}$, for
some $v\in \Gamma$, then $NK_{v}=(0)$, a contradiction. If
$N_{p}K_{q}=N$ or $N_{p}K_{q}=K$, then $N^{2} =(0)$ or $K^{2}
=(0)$,
respectively and the claim is proved.\\\\
Here, without loss of generality, we suppose that $N^{2} =
(0)$. Clearly, $(N:M)M\nsubseteq K$ and $(K:M)M\nsubseteq N$.\\\\
Claim 2. Our claim is to show that $N$ is a minimal submodule of
$M$ and $K^{2}\neq (0)$. To see that, first we show that for every
$0\neq m\in N$, $Rm=N$. Assume that $0\neq m\in N$ and $Rm\neq N$.
If $Rm = K$, then $K\subseteq N$, a contradiction. Thus $Rm\neq
K$, and the induced subgraph of $AG(M)$ on $N$, $K$, and $Rm$ is
$K_{3}$, a contradiction. So $Rm = N$. This implies that $N$ is a
minimal submodule of $M$. Now if $K^{2}=(0)$, then we obtain the
induced subgraph on $N$, $K$, and
$(N:M)M+(K:M)M$ is $K_{3}$, a contradiction. Thus $K^{2}\neq (0)$, as desired.\\\\
Claim 3. For every $i\in \Lambda$ and every $j\in \Gamma$,
$N_{i}\cap K_{j}=N$. Let $i\in \Lambda$ and $j\in \Gamma$. Since
$N_{i}\cap K_{j}$ is a vertex and $N(N_{i}\cap K_{j})=K(N_{i}\cap
K_{j})=(0)$, either $N_{i}\cap K_{j}=N$ or $N_{i}\cap K_{j}=K$. If
$N_{i}\cap K_{j}=K$, then $K^{2}=(0)$, a contradiction. Hence
$N_{i}\cap K_{j}=N$ and the claim is
proved.\\\\
Claim 4. We complete the claim by showing that $M$ has exactly two
minimal submodules $N$ and $K$. Let $L$ be a non-zero submodule
properly contained in $K$. Since $NL\subseteq NK=(0)$, either
$L=N$ or $L=N_{i}$ for some $i\in \Lambda$. So by the Claim 3,
$N\subseteq L\subseteq K$, a contradiction. Hence $K$ is a minimal
submodule of $M$. Suppose that $L'$ is another minimal submodule
of $M$. Since $N$ and $K$ both are minimal submodules, we deduce
that $NL'=KL'=(0)$, a
contradiction. So the claim is proved.\\\\
Now by the Claims 2 and 4, $K^{2}\neq (0)$ and $K$ is a minimal
submodule of $M$. Then by Lemma \ref{l2.4}, $K=eM$ for some
idempotent $e\in R$. Now we have $M\cong eM\times (1-e)M$. By
Lemma \ref{l2.6}, we deduce that either $M=F \times S$ and
$AG(M)\cong P_{4}$ or $R=F \times D$ and $AG(M)$ is a star graph.
Conversely, we assume that $M=F \times S$. Then $AG(M)$ has
exactly four vertices $(0)\times S$, $F \times (0)$, $(0)\times
N$, and $F\times N$. Thus $AG(M)\cong P_{4}$ with the vertices
$(0)\times S$, $F \times (0)$, $(0)\times N$, and $F\times N$.
\end{proof}

\begin{thm}\label{t2.8}
Let $R$ be an Artinian ring and $AG(M)$ is a bipartite graph. Then
either $AG(M)$ is a star graph or $AG(M)\cong P_{4}$. Moreover,
$AG(M)\cong P_{4}$ if and only if $M=F \times S$, where $F$ is a
simple module and $S$ is a module with a unique non-trivial
submodule.
\end{thm}

\begin{proof}
First suppose that $R$ is not a local ring. Hence by \cite[Theorem
8.9]{ati69}, $R = R_{1}\times \ldots \times R_{n}$, where $R_{i}$
is an Artinian local ring for $i=1, \ldots , n$. By Lemma
\ref{l2.2} and Proposition \ref{p2.1}, since $AG(M)$ is a
bipartite graph, we have $n=2$ and hence $M\cong M_{1}\times
M_{2}$. If $M_{1}$ is a prime module, then it is easy to see that
$M_{1}$ is a vector space over $R/Ann(M_{1})$ and so is a
semisimple $R$-module. Hence by Lemma \ref{l2.6} and Theorem
\ref{t2.7}, we deduce that either $AG(M)$ is isomorphic to $P_{2}$
or $P_{4}$. Now we assume that $R$ is an Artinian local ring. Let
$m$ be the unique maximal ideal of $R$ and $k$ be a natural number
such that $m^{k}M=0$ and $m^{k-1}M\neq 0$. Clearly, $m^{k-1}M$ is
adjacent to every other vertex of $AG(M)$ and so $AG(M)$ is a star
graph.
\end{proof}

\begin{prop}\label{p2.9} Assume that $Ann(M)$ is a nil
ideal of $R$.

\begin{itemize}
\item [(a)] If $AG(M)$ is a finite bipartite graph, then either
$AG(M)$ is a star graph or $AG(M)\cong P_{4}$. \item [(b)] If
$AG(M)$ is a regular graph of finite degree, then $AG(M)$ is a
complete graph.
\end{itemize}

\end{prop}

\begin{proof}
$(a)$. If $M$ is a vertex of $AG(M)$, then $AG(M)$ has only one
vertex $N$ such that $Ann(M)=(N:M)$ and since $AG(M)^{*}$ is an
empty subgraph, $AG(M)$ is a star graph. Thus we may assume that
$M$ is not a vertex of $AG(M)$ and hence by \cite[Theorem
3.3]{ah14}, $M$ is not a prime module. Therefore \cite[Theorem
3.6]{ah14} follows that $R/Ann(M)$ is an Artinian ring. If
$(R/Ann(M), m/Ann(M))$ is a local ring, then there exists a
natural number $k$ such that $m^{k}M=0$ and $m^{k-1}M\neq 0$.
Clearly, $m^{k-1}M$ is adjacent to every other vertex of $AG(M)$
and so $AG(M)$ is a star graph. Otherwise, by \cite[Theorem
8.9]{ati69} and Lemma \ref{l2.2}, there exist pairwise orthogonal
central idempotents modulo $Ann(M)$. By Lemma \ref{l2.3}, it is
easy to see that $M\cong eM\times (1-e)M$, where $e$ is an
idempotent element of $R$ and Lemma \ref{l2.6} implies that
$AG(M)$ is a star graph or
$AG(M)\cong P_{4}$.\\
$(b)$. If $M$ is a vertex of $AG(M)$, since $AG(M)$ is a regular
graph, then $AG(M)$ is a complete graph. Hence we may assume that
$M$ is not a vertex of $AG(M)$. So $M$ is not a prime module, and
hence $rm=0$ such that $0\neq m\in M$, $r\notin Ann(M)$. It is
easy to see that $(rM)(0:_{M}r)=(0)$. If the set of $R$-submodules
of $rM$ (resp., $(0:_{M}r)$)) is infinite, then $(0:_{M}r)$
(resp., $rM$) has infinite degree, a contradiction. Thus $rM$ and
$(0:_{M}r)$ have finite length. Since $rM\cong M/(0:_{M}r)$, $M$
has finite length so that $R/Ann(M)$ is an Artinian ring. As in
the proof of part (a), $M\cong M_{1}\times M_{2}$. If $M_{1}$ has
one non-trivial submodule $N$, then $deg((0) \times M_{2}) > deg(N
\times M_{2})$ and this contradicts the regularity of $AG(M)$.
Hence, $M_{1}$ is a simple module. Similarly, $M_{2}$ is a simple
module. So $AG(M)\cong K_{2}$. Now suppose that
$(R/Ann(M),m/Ann(M))$ is an Artinian local ring. Now as we have
seen in part (a), there exists a natural number $k$ such that
$m^{k-1}M$ is adjacent to all other vertices and we deduce that
$AG(M)$ is a complete graph.

\end{proof}

Let $S$ be a multiplicatively closed subset of $R$. A non-empty
subset $S^{*}$ of $M$ is said to be $S$-closed if $se\in S^{*}$
for every $s\in S$ and $e\in S^{*}$. An $S$-closed subset $S^{*}$
is said to be saturated if the following condition is satisfied:
whenever $ae\in S^{*}$ for $a\in R$ and $e\in M$, then $a\in S$
and $e\in S^{*}$.

We need the following result due to Chin-Pi Lu.

\begin{thm}\label{t2.10} $($See \cite[Theorem 4.7]{lu97}.$)$
Let $M=Rm$ be a cyclic module. Let $S^{*}$ be an $S$-closed subset
of $M$ relative to a multiplicatively closed subset $S$ of $R$,
and $N$ a submodule of $M$ maximal in $M\setminus S^{*}$. If
$S^{*}$ is saturated, then ideal $(N:M)$ is maximal in $R\setminus
S$ so that $N$ is prime in $M$.
\end{thm}

\begin{thm}\label{t2.11}
If $M$ is a cyclic module, $Ann(M)$ is a nil ideal, and
$|Min(M)|\geq 3$, then $AG(M)$ contains a cycle.
\end{thm}

\begin{proof}
If $AG(M)$ is a tree, then by Theorem \ref{t2.7}, either $AG(M)$
is a star graph or $M\cong F \times S$, where $F$ is a simple
module and $S$ has a unique non-trivial submodule. The latter case
is impossible because $|Min(F \times S)| = 2$. Suppose that
$AG(M)$ is a star graph and $N$ is the center of star. Clearly,
one can assume that $N$ is a minimal submodule of $M$. If
$N^{2}\neq (0)$, then by Lemma \ref{l2.4}, there exists an
idempotent $e\in R$ such that $N=eM$ so that $M\cong eM \times (1
- e)M$. Now by Proposition \ref{p2.1} and Lemma \ref{l2.6}, we
conclude that $|Min(M)|=2$, a contradiction. Hence $N^{2}= 0$.
Thus one may assume that $N = Rm$ and $(Rm)^{2} = (0)$. Suppose
that $P_{1}$ and $P_{2}$ are two distinct minimal prime submodules
of $M$. Since $(Rm)^{2} = (0)$, we have $(Rm:M)^{2}\subseteq
Ann(M)\subseteq (P_{i}:M)$, $i=1,2$. So $(Rm:M)M=Rm\subseteq
P_{i}$, $i=1,2$. Hence $m\in P_{i}$, $i=1,2$. Choose $z\in
(P_{1}:M)\setminus (P_{2}:M)$ and set $S_{1}=\{1, z, z^{2}, \ldots
\}$, $S_{2}=M\setminus P_{1}$, and $S^{*}=S_{1}S_{2}$. If $0\notin
S^{*}$, then $\Sigma =\{N< M|$ $N\cap S^{*}=\emptyset\}$ is not
empty. Then $\Sigma$ has a maximal element, say $N$. Hence by
Theorem \ref{t2.10}, $N$ is a prime submodule of $M$. Since
$N\subseteq P_{1}$, we have $N=P_{1}$, a contradiction because
$z\notin (N:M)$. So $0\in S^{*}$. Therefore, there exist positive
integer $k$ and $m'\in S_{2}$ such that $z^{k}m'=0$. Now consider
the submodules $(m), (m')$, and $z^{k}M$. It is clear that
$(m)\neq (m')$ and $(m)\neq z^{k}M$. If $(m)=z^{k}M$, then $z\in
(P_{2}:M)$, a contradiction. Thus $(m), (m')$, and $z^{k}M$ form a
triangle in $AG(M)$, a contradiction. Hence $AG(M)$ contains a
cycle.
\end{proof}

\begin{thm}\label{t2.12}
Suppose that $M$ is a cyclic module, $rad_{M}(0)\neq (0)$, and
$Ann(M)$ is a nil ideal. If $|Min(M)|=2$, then either $AG(M)$
contains a cycle or $AG(M)\cong P_{4}$.
\end{thm}

\begin{proof}
A similar argument to the proof of Theorem \ref{t2.11} shows that
either $AG(M)$ contains a cycle or $M\cong F\times S$, where $F$
is a simple module and $S$ is a module with a unique non-trivial
submodule. The latter case implies that $AG(M)\cong P_{4}$ (note
that $rad_{F\times D}(0)=(0)$, where $F$ is a simple module and
$D$ is a prime module).
\end{proof}

We recall that $N< M$ is said to be a semiprime submodule of $M$
if for every ideal $I$ of $R$ and every submodule $K$ of $M$,
$I^{2}K\subseteq N$ implies that $IK\subseteq N$. Further $M$ is
called a semiprime module if $(0)\subseteq M$ is a semiprime
submodule. Every intersection of prime submodules is a semiprime
submodule (see \cite{tv08}).

\begin{thm}\label{t2.13}
Let $S$ be a maltiplicatively closed subset of $R$ containing no
zero-divisors on finitely generated module $M$. Then
$cl(AG(M_{S}))\leq cl(AG(M))$. Moreover, $AG(M_{S})$ is a retract
of $AG(M)$ if $M$ is a semiprime module. In particular,
$cl(AG(M_{S}))= cl(AG(M))$, whenever $M$ is a semiprime module.
\end{thm}

\begin{proof}
Consider a vertex map $\phi: V(AG(M)) \longrightarrow
V(AG(M_{S})), N\longrightarrow N_{S}$. Clearly, $N_{S}\neq K_{S}$
implies $N\neq K$ and $NK = (0)$ if and only if $N_{S} K_{S}=
(0)$. Thus $\phi$ is surjective and hence $cl(AG(M_{S}))\leq
cl(AG(M))$. In what follows, we assume that $M$ a semiprime
module. If $N \neq K$ and $NK = (0)$, then we show that $N_{S}\neq
K_{S}$. Without loss of generality we can assume that $M$ is not a
vertex of $AG(M)$ and On the contrary suppose that $N_{S}= K_{S}$.
Then $N_{S}^{2} =N_{S}K_{S}=(NK)_{S} =(0)$ and so $N^{2}=(0)$, a
contradiction. This shows that the map $\phi$ is a graph
homomorphism. Now, for any vertex  $N_{S}$ of $AG(M_{S})$, we can
choice the fixed vertex $N$ of $AG(M)$. Then $\phi$ is a retract
(graph) homomorphism which clearly implies that $cl(AG(M_{S}))=
cl(AG(M))$ under the assumption.
\end{proof}

\begin{cor}\label{c2.14}
If $M$ is a  finitely generated semiprime module, then
$cl(AG(T(M))= cl(AG(M))$, where $T=R\setminus Z(M)$.
\end{cor}

Since the chromatic number $\chi(G)$ of a graph $G$ is the least
positive integer $r$ such that there exists a retract homomorphism
$\psi: G\longrightarrow K_{r}$, the following corollaries follow
directly from the proof of Theorem \ref{t2.13}.

\begin{cor}\label{c2.15}
Let $S$ be a maltiplicatively closed subset of $R$ containing no
zero-divisors on finitely generated module $M$. Then
$\chi(AG(M_{S}))\leq \chi(AG(M))$. Moreover, if $M$ is a semiprime
module, then $\chi(AG(M_{S}))= \chi(AG(M))$.
\end{cor}

\begin{cor}\label{c2.16}
If $M$ is a finitely generated semiprime module, then
$\chi(AG(T(M))= \chi(AG(M))$, where $T=R\setminus Z(M)$.
\end{cor}

Eben Matlis in \cite[Proposition 1.5]{m83}, proved that if
$\{p_{1}, \ldots, p_{n}\}$ is a finite set of distinct minimal
prime ideals of $R$ and $S = R\setminus \cup^{n}_{i=1} p_{i}$,
then $R_{p_{1}}\times \ldots \times R_{p_{n}}\cong R_{S}$. In
\cite {s11}, this result was generalized to finitely generated
multiplication modules. In Theorem \ref{t2.18}, we use this
generalization for a cyclic module.

\begin{thm}\label{t2.17} $($See \cite[Theorem 3.11]{s11}.$)$
Let $\{P_{1}, \ldots, P_{n}\}$ be a finite set of distinct minimal
prime submodules of finitely generated multiplication module $M$
and $S = R\setminus \cup^{n}_{i=1} (P_{i}:M)$. Then
$M_{p_{1}}\times \ldots \times M_{p_{n}}\cong M_{S}$, where
$p_{i}=(P_{i}:M)$ for $1\leq i\leq n$.
\end{thm}

\begin{thm}\label{t2.18}
Let $M$ be a cyclic module and $\{P_{1}, \ldots, P_{n}\}$ be a
finite set of distinct minimal prime submodules of $M$. Then there
exists a clique of size $n$.
\end{thm}

\begin{proof}
Let $M$ be a cyclic module and $S = R\setminus \cup^{n}_{i=1}
p_{i}$, where $p_{i}=(P_{i}:M)$ for $1\leq i\leq n$. Then since
$M$ is a multiplication module, by Theorem \ref{t2.17}, there
exists an isomorphism $\phi : M_{p_{1}}\times \ldots \times
M_{p_{n}}\longrightarrow M_{S}$. Let $M=Rm, e_{i} = (0, \ldots ,
0, m/1, \ldots , 0, \ldots, 0 )$ and $\phi(e_{i})=n_{i}/t_{i}$ ,
where $m\in M$, $1\leq i \leq n$, and $m/1$ is in the $i$-th
position of $e_{i}$. Consider the principal submodules
$N_{i}=(n_{i}/t_{i})=(n_{i}/1)$ in the module $M_{S}$. By Lemma
\ref{l2.2} and Proposition \ref{p2.1}, the product of submodules
$(0)\times \ldots \times (0)\times (m/1)R_{p_{i}}\times (0)\times
\ldots \times (0)$ and $(0)\times \ldots \times (0)\times
(m/1)R_{p_{j}}\times (0)\times \ldots \times (0)$ are zero, $i\neq
j$. Since $\phi$ is an isomorphism, there exists $t_{ij}\in S$
such that $t_{ij}r_{i}n_{j}=0$, for every $i, j, 1\leq i<j\leq n$,
where $n_{i}=r_{i}m$ for some $r_{i}\in R$. Let $t=\Pi_{1\leq
i<j\leq n} t_{ij}$. We show that $\{(tn_{1}), \ldots , (tn_{n})\}$
is a clique of size $n$ in $AG(M)$. For every $i, j, 1\leq i<j\leq
n$, $(Rtn_{i})(Rtn_{j})=(Rtn_{j}:M)Rtn_{i}=(Rtn_{j}:M)tr_{i}M=
tr_{i}Rtn_{j}=(0)$. Since $(tn_{i})_{S}=(n_{i}/1)=N_{i}$, we
deduce that $(tn_{i})$ are distinct non-trivial submodules of $M$.
\end{proof}

\begin{cor}\label{c2.19}
For every cyclic module $M$, $cl(AG(M))\geq |Min(M)|$ and if
$|Min(M)|\geq 3$, then $gr(AG(M))=3$.
\end{cor}

\begin{thm}\label{t2.20}
Let $M$ be a cyclic module and $rad_{M}(0)=(0)$. Then $\chi(AG(M))
= cl(AG(M)) = |Min(M)|$.
\end{thm}

\begin{proof}
If $|Min(M)|=\infty$, then by Corollary \ref{c2.19}, there is
nothing to prove. Thus suppose that $|Min(M)|=\{P_{1}, ... ,
P_{n}\}$, for some positive integer $n$. Let $p_{i}=(P_{i}:M)$ and
$S = R\setminus \cup^{n}_{i=1} p_{i}$. By  Theorem \ref{t2.17}, we
have $M_{p_{1}}\times \ldots \times M_{p_{n}}\cong M_{S}$.
Clearly, $cl(AG(M_{S}))\geq n$. Now we show that
$\chi(AG(M_{S}))\leq n$. By \cite[Corollary 3]{lu95},
$P_{i}R_{p_{i}}$ is the only prime submodule of $M$ and since
$rad_{M}(0)=(0)$, every $M_{p_{i}}$ is a simple
$R_{p_{i}}$-module. Define the map $C:V(AG(M_{S})) \longrightarrow
\{1, 2, \ldots , n\}$ by $C(N_{1} \times \ldots \times N_{n})=min
\{i|$  $N_{i}\neq (0)\}$. Since each $M_{p_{i}}$ is a simple
module, $c$ is a proper vertex coloring of $AG(M_{S})$. Thus
$\chi(AG(M_{S}))\leq n$ and so $\chi(AG(M_{S})) = cl(AG(M_{S})) =
n$. Since $rad_{M}(0)=(0)$, it is easy to see that $S\cap
Z(M)=\emptyset$. Now by theorem \ref{t2.13} and Corollary
\ref{c2.15}, we obtain the desired.
\end{proof}

\begin{thm}\label{t2.21}
For every module $M$,  $cl(AG(M)) = 2$ if and only if
$\chi(AG(M))=2$.
\end{thm}

\begin{proof}
For the first assertion, we use the same technique in
\cite[Theorem 13]{aa14}. Let $cl(AG(M)) = 2$. On the contrary
assume that $AG(M)$ is not bipartite. So $AG(M)$ contains an odd
cycle. Suppose that $C:= N_{1} - N_{2} - \ldots -  N_{2k+1} -
N_{1}$ be a shortest odd cycle in $AG(M)$ for some natural number
$k$. Clearly, $k\geq 2$. Since $C$ is a shortest odd cycle in
$AG(M)$, $N_{3}N_{2k+1}$ is a vertex. Now consider the vertices
$N_{1}, N_{2}$, and $N_{3}N_{2k+1}$. If $N_{1}=N_{3}N_{2k+1}$,
then $N_{4} N_{1}= (0)$. This implies that $ N_{1} - N_{4} -
\ldots - N_{2k+1} - N_{1}$ is an odd cycle, a contradiction. Thus
$N_{1}\neq N_{3}N_{2k+1}$. If $N_{2}= N_{3}N_{2k+1}$, then we have
$C_{3}= N_{2} - N_{3}$ - $N_{4} - N_{2}$, again a contradiction.
Hence $N_{2}\neq N_{3}N_{2k+1}$. It is easy to check $N_{1},
N_{2}$, and $N_{3}N_{2k+1}$ form a triangle in $AG(M)$, a
contradiction. The converse is clear.
\end{proof}

The radical of $I$, defined as the intersection of all prime
ideals containing $I$, denoted by $\sqrt{I}$. Before stating the
next theorem, we recall that if $M$ is a finitely generated
module, then $\sqrt{(Q:M)}=(rad(Q):M)$, where $Q< M$ (see \cite
{lu90} and \cite[Proposition 2.3]{lu10}). Also, we know that if
$M$ is a finitely generated module, then for every prime ideal $p$
of $R$ with $p\supseteq Ann(M)$, there exists a prime submodule
$P$ of $M$ such that $(P:M)=p$ (see \cite[Theorem 2]{lu95}).

\begin{thm}\label{t2.22}
Assume that $M$ is a finitely generated module, $Ann(M)$ is a nil
ideal, and $|Min(M)| = 1$. If $AG(M)$ is a triangle-free graph,
then $AG(M)$ is a star graph.
\end{thm}

\begin{proof}
Suppose first that $P$ is the unique minimal prime submodule of
$M$. Since $M$ is not a vertex of $AG(M)$, hence $Z(M)\neq (0)$.
So there exist non-zero elements $r\in R$ and $m\in M$ such that
$rm=0$. It is easy to see that $rM$ and $Rm$ are vertices of
$AG(M)$ because $(rM)(Rm)=0$. Since $AG(M)$ is triangle-free, $Rm$
or $rM$ is a minimal submodule of $M$. Without loss of generality,
we can assume that $Rm$ is a minimal submodule of $M$ so that
$(Rm)^{2}=(0)$ (if $rM$ is a minimal submodule of $M$, then there
exists $0\neq m'\in M$ such that $rM=Rm'$). We claim that $Rm$ is
the unique minimal submodule of $M$. On the contrary, suppose that
$K$ is another minimal submodule of $M$. So either $K^{2} = K$ or
$K^{2} = 0$. If $K^{2} = K$, then by Lemma \ref{l2.4}, $K = eM$
for some idempotent element $e\in R$ and hence $M \cong eM \times
(1-e)M$. This implies that $|Min(M)| > 1$, a contradiction. If
$K^{2} = 0$, then we have $C_{3}= K - (K:M)M + (Rm:M)M - Rm - K$,
a contradiction. So $Rm$ is the unique minimal submodule of $M$.
Let $V_{1}=N((Rm))$, $V_{2}=V(AG(M))\setminus V_{1}$, $A=\{K\in
V_{1}|(Rm)\subseteq K\}$, $B=V_{1}\setminus A$, and
$C=V_{2}\setminus
 \{Rm\}$. We prove
that $AG(M)$ is a bipartite graph with parts $V_{1}$ and $V_{2}$.
We may assume that $V_{1}$ is an independent set because $AG(M)$
is triangle-free. We claim that one end of every edge of $AG(M)$
is adjacent to $Rm$ and another end contains $Rm$. To prove this,
suppose that $\{N, K\}$ is an edge of $AG(M)$ and $Rm \neq N$, $Rm
\neq K$. Since $N(Rm)\subseteq Rm$, by the minimality of $Rm$,
either $N(Rm)=(0)$ or $Rm\subseteq N$. The latter case follows
that $K(Rm)=(0)$. If $N(Rm)=(0)$, then $K(Rm)\neq (0)$ and hence
$Rm\subseteq K$. So our plain is proved. This gives that $V_{2}$
is an independent set and $N(C) \subseteq V_{1}$. Since every
vertex of $A$ contains $Rm$ and $AG(M)$ is triangle-free, all
vertices in $A$ are just adjacent to $Rm$ and so by \cite[Theorem
3.4]{ah14}, $N(C) \subseteq B$. Since one end of every edge is
adjacent to $Rm$ and another end contains $Rm$, we also deduce
that every vertex of $C$ contains $Rm$ and so every vertex of $A
\cup V_{2}$ contains $Rm$. Note that if $Rm = P$, then one end of
each edge of $AG(M)$ is contained in $Rm$ and since $Rm$ is a
minimal submodule of $M$, $AG(M)$ is a star graph with center
$Rm=P$. Now, suppose that $P\neq Rm$. We claim that $P\in A$.
Since $Rm\subseteq P$, it suffices to show that $(Rm)P=(0)$. To
see this, let $r\in (P:M)$. We prove that $rm = 0$. Clearly,
$(Rrm)\subseteq Rm$. If $rm = 0$, then we are done. Thus $Rrm =
Rm$ and so $m=rsm$ for some $s\in R$. We have $m(1-rs) = 0$. By
\cite[Theorem 2]{lu95}, we have $Nil(R)=(P:M)$ (note that
$\sqrt{Ann(M)}=(rad(0):M)=(P:M)$). Therefore $1-rs$ is unit, a
contradiction, as required. Since $N(C)\subseteq B$, if $B =
\emptyset$, then $C = \emptyset$ and so $AG(M)$ is a star graph
with center $Rm$. It remains to show that $B = \emptyset$. Suppose
that $K\in B$ and consider the vertex $K \cap P$ of $AG(M)$. Since
every vertex of $A\cup V_{2}$ contains $Rm$, yields $K \cap P\in
B$. Pick $0 \neq m'\in K \cap P$. Since $AG(M)$ is triangle-free,
one can find an element $m''\in Rm'$ such that $Rm''$ is a minimal
submodule of $M$ and $(Rm'')^{2} = (0)$. Since $Rm$ is the unique
minimal submodule of $M$, we have $Rm = Rm''\subseteq Rm'$. Thus
$Rm\subseteq K\cap P$, a contradiction. So $B = \emptyset$ and we
are done. Hence $AG(M)$ is a star graph whose center is $Rm$, as
desired.
\end{proof}

\begin{cor}\label{c2.23}
Assume that $M$ is a finitely generated module, $Ann(M)$ is a nil
ideal, and $|Min(M)| = 1$. If $AG(M)$ is a bipartite graph, then
$AG(M)$ is a star graph.
\end{cor}

\vspace{2mm} \noindent \footnotesize
\begin{minipage}[b]{10cm}
Department of pure Mathematics \\
Faculty of mathematical Sciences, \\
University of Guilan, P. O. Box 41335-19141, Rasht, Iran \\
E-mail: ansari@guilan.ac.ir
\end{minipage}

\vspace{2mm} \noindent \footnotesize
\begin{minipage}[b]{10cm}
Department of pure Mathematics \\
Faculty of mathematical Sciences, \\
University of Guilan, P. O. Box 41335-19141, Rasht, Iran \\
E-mail: sh.habibi@phd.guilan.ac.ir
\end{minipage}

\end{document}